\documentclass[11pt]{article}
\usepackage{latexsym}

\usepackage[hidelinks]{hyperref}

\usepackage{amsmath}
\usepackage{amsthm}
\usepackage{amssymb}
\usepackage{mathrsfs}
\usepackage{wasysym,color}
\usepackage{graphicx}
\usepackage[title]{appendix}
\usepackage{comment}
\usepackage{bm}
\usepackage{amsfonts}
\usepackage{graphicx}   
\usepackage{relsize}    
\usepackage{float}

\def\c\delta{\mathcal{\delta}}

\newtheorem{theorem}{Theorem}

\newtheorem{lemma}[theorem]{Lemma}

\theoremstyle{definition}
\newtheorem{definition}[theorem]{Definition}

\allowdisplaybreaks[4]

\numberwithin{equation}{section}
\numberwithin{theorem}{section}

\begin{document}
\centerline{\Large\bf Mean-square attractors for non-autonomous Caputo }

\centerline{\Large\bf fractional stochastic differential equations
\footnote{This work was supported by the National Natural Science Foundation of China under grant 12371198.}}

\vskip 6pt \centerline{{\bf Lijuan Zhang$^{1}$,}~~  ~~{\bf Jianhua Huang\footnote{Corresponding author: Jianhua Huang.\\
\indent
E-mail addresses: zhanglj@nau.edu.cn,  jhhuang32@nudt.edu.cn, wangyj@lzu.edu.cn},}
~~  ~~{\bf Yejuan Wang$^{2}$}} \vskip 3pt

\centerline{\footnotesize\it $^1$School of Mathematics, Nanjing Audit University, Nanjing 210000, China}
\centerline{\footnotesize\it $^{\dag}$College of Science, National University of Defense Technology, Changsha 410073, China}
\centerline{\footnotesize\it $^2$School of Mathematics and Statistics,
Lanzhou University, Lanzhou 730000, China}



\bigbreak \noindent{\bf Abstract}{
This paper investigates the existence of mean-square attractors for a class of non-autonomous Caputo fractional stochastic differential equations of order $\alpha\in (\frac{1}{2},1)$, with a driving system on a compact base space $P$ and tempered fractional noise.
We first construct a mean-square semi-dynamical system on $\mathfrak{C} \times P$ that carries a skew-product semi-flow structure,
where $\mathfrak{C}=C(\mathbb{R}^{+},
L^{2}(\Omega, \mathcal{F}; \mathbb{R}^d))$ denotes the space of continuous functions from $ \mathbb{R}^{+}$ into $L^2(\Omega, \mathcal{F}; \mathbb{R}^d)$.
A global forward attracting set is then established in the weak mean-square topology. Moreover, by endowing the function space  $\mathfrak{C}_{w}=C(\mathbb{R}^{+},
L^{2}_w(\Omega, \mathcal{F}; \mathbb{H}))$ with an appropriate topology that renders it complete, we show that the skew-product semi-flow possesses a bounded and closed mean-square attractor within  $\mathfrak{C}_{w} \times P$. It is worth emphasizing that completeness plays a crucial role here: without this property, the attractor need not exist.



\noindent{\section{Introduction}}

Fractional stochastic differential equations provide a unified mathematical framework for modeling anomalous diffusion phenomena subject to stochastic perturbations.
Stochastic differential equations with fractional substantial derivative possess non-local characteristics, which can effectively characterize the long-term dependence of the current state of a system on its historical trajectory (i.e., the power-law memory effect).
Systematic investigation of their dynamical evolution mechanisms can advance the theory of fractional stochastic dynamical systems and provide a theoretical support for non-Markovian and non-Gaussian stochastic phenomena in interdisciplinary fields.

The Caputo fractional substantial time derivative \cite{Podlubny} with order $\alpha$ and decay parameter $\varrho>0$ is defined by
\[
D^{\alpha,\varrho} f(t) = I^{\beta,\varrho} [D^{m,\varrho} f(t)], \quad \alpha = m - \beta,
\]
where $m = \lceil \alpha \rceil$ is the smallest integer greater than $\alpha$.  Here the operator $I^{\beta,\varrho}$ denotes the fractional substantial integral \cite{Chen-Deng-2015-1, Deng-Chen-2015-2} given by
\[
I^{\beta,\varrho} f(t) = \frac{1}{\Gamma(\beta)} \int_0^t (t - \tau)^{\beta-1} e^{-\varrho(t-\tau)} f(\tau) d\tau, \quad \beta > 0,
\]
where $\varrho$ may be either constant or a function of the underlying variables. The operator $D^{m, \varrho}$ is understood as the $m$-fold composition $D^{m, \varrho}=( \frac{\partial}{\partial t} + \varrho)^m$, that is, the operator $(\frac{\partial}{\partial t} + \varrho)$ applied successively $m$ times.

Mean-square attractors, introduced by Kloeden and Lorenz \cite{Kloeden-Lorenz} in 2012, provide a natural framework for capturing the long-term statistical behavior of stochastic systems in the mean-square sense. Wu and Kloeden \cite{Wu-Kloeden} established the existence of a random attractor for the mean-square random dynamical system generated by a stochastic delay equation. Subsequently,
Wang \cite{B-Wang} proved the existence of weak pullback mean random attractors for mean random dynamical systems in Bochner spaces.
Doan et al. \cite{Doan-Rasmussen-Kloeden} analyzed the dichotomy spectrum for linear mean-square random dynamical systems and showed that a change in the sign of the dichotomy spectrum implies a bifurcation from a trivial to a nontrivial mean-square random attractor.  Meanwhile, mean-square attractors for stochastic lattice systems were studied in \cite{Kloeden-Lorenz-2}, while the mean-square controllability and continuity of pullback random attractors with respect to the Hausdorff metric were investigated in \cite{Liu-Li-Tang}.

To date, the investigation of dynamical systems associated with Caputo fractional differential equations remains relatively limited.
To the best of our knowledge, Doan and Kloeden \cite{Doan-Kloeden} showed that an autonomous Caputo fractional differential equation in a finite-dimensional space generates a semi-dynamical system in the function space. In \cite{Doan-Kloeden-2022}, the attractor of an autonomous Caputo fractional differential equation in \( \mathbb{R}^d \) was considered, where the vector field possesses a certain triangular structure and satisfies both a smoothness condition and a dissipativity condition. Additionally, Doan and Kloeden \cite{Doan-Kloeden-2024} demonstrated that an autonomous Caputo fractional differential equation in \( \mathbb{R}^d \) has a global Caputo attractor in the space of continuous functions. For non-autonomous systems, Cui and Kloeden \cite{Cui-Kloeden-2024} proved that a non-autonomous Caputo fractional differential equation in \( \mathbb{R}^d \), with a driving system on a compact base space, generates a skew-product semi-flow; furthermore, this skew-product semi-flow was shown to admit a bounded, closed attractor.

As far as we know, the theory of dynamics for non-autonomous Caputo fractional stochastic differential equations  (FSDEs) remains unexplored. Motivated by both its practical value in effectively modeling real-world natural phenomena with random perturbations, and its academic necessity in filling the existing theoretical gap, the investigation of dynamics in this field is of great importance.
The main difficulty lies in that for Caputo FSDEs driven by tempered fractional Brownian motion, the non-Markovian and non-martingale properties of the random noise, together with the nonlocality of the Caputo fractional derivative, result in the evolution of solutions depending not only on the current state but also on the entire history of the system. Consequently, it becomes challenging to establish a semi-group in the autonomous case or a two-parameter semi-group in the non-autonomous case on $L^{2}(\Omega, \mathcal{F}; \mathbb{R}^d)$.

This paper aims to fill this gap by developing a systematic theory for mean-square dynamical systems governed by the following non-autonomous Caputo FSDE driven by tempered fractional Brownian motion:
\begin{align}\label{eq00.2-2}
D^{\alpha,\varrho} x(t) =g(x(t), \vartheta_t(p))+\frac{dB^{H,\lambda}(t)}
{dt},~~~t\geq0,
\end{align}
with initial values $x_{0}\in L^2(\Omega, \mathcal{F}_{0};\mathbb{R}^d)$ and $p_{0}\in P$.
Here $D^{\alpha,\varrho}$
denotes the Caputo fractional substantial derivative with order $\alpha\in(\frac{1}{2},1)$ and decay parameter $\varrho>0$;
$P$ is a compact base space, and
$\vartheta_t: P\rightarrow P$, $t\in \mathbb{R}$ forms a group of operators, i.e., an autonomous dynamical system.
The non-autonomous function $g: \mathbb{R}^d\times P \to \mathbb{R}^d$ is assumed to be Lipschitz continuous.
Moreover,
$B^{H,\lambda}(t)=(B^{H_{1},\lambda_{1}}_{1}(t),\ldots, B^{H_{d},\lambda_{d}}_{d}(t))^{\top}$ denotes a $d$-dimensional tempered fractional Brownian motion (TFBM) \cite{Meerschaert-Sabzikar-2013, Meerschaert-Sabzikar-2014} with Hurst index $H=(H_{1},\ldots, H_{d})$ and tempering parameter $\lambda=(\lambda_{1},\ldots, \lambda_{d})$. Each component $B^{H_{i},\lambda_{i}}_{i}(t)$,  $i=1,\ldots,d$, is an independent one-dimensional TFBM with Hurst parameter $H_{i}\in(\frac{1}{2},1)$ and tempering parameter $\lambda_{i}>0$.

The solution of the Caputo FSDE \eqref{eq00.2-2} admits the following integral representation:
\begin{align}
    x(t, x_0, p_0) = x_0 e^{-\varrho t}+ \int_0^t a(t, s) g(x(s), \vartheta_s(p_0))ds
    +\int_0^t a(t, s)dB^{H,\lambda}(s),\nonumber
\end{align}
where
\begin{align}\label{eq00.2-1}
a(t,s):=\frac{1}{\Gamma(\alpha)}(t-s)^{\alpha-1}e^{-\varrho (t-s)}
\end{align}
is the memory kernel associated with the Caputo fractional substantial derivative.
The positive parameter $\varrho$ governs the exponential decay rate of this kernel, thereby quantifying the fading influence of historical states on the current evolution.

The nonlocal nature of the Caputo FSDE \eqref{eq00.2-2} prevents the direct generation of a dynamical system on the state spaces $\mathbb{R}^{d}$ or $L^{2}(\Omega, \mathcal{F}; \mathbb{R}^{d})$. Accordingly, we construct a family of operators
$\{T_{t,\xi}\}_{t,\xi\in \mathbb{R}^+}$ on the continuous function space  $C(\mathbb{R}^+, L^{2}(\Omega, \mathcal{F}; \mathbb{R}^{d}))$.
By using the translation invariance of the Caputo kernel function $a(t, s)$, together with the stationarity and It\^{o} isometry of TFBM, we verify that
$\{T_{t,\xi}\}_{t,\xi\in \mathbb{R}^+}$ satisfies the cocycle property. In conjunction with the compact base space $P$, this allows us to prove that the non-autonomous Caputo FSDE \eqref{eq00.2-2} generates a mean-square semi-dynamical system $\{\Pi_{t,\xi}=
(T_{t,\xi}, \vartheta_t)\}_{t,\xi\in \mathbb{R}^+}$ with a skew-product semi-flow structure, as stated in Theorem \ref{thm1.1}.
Subsequently, we establish the uniform boundedness of solutions with respect to $p\in P$ (Theorem~\ref{thm1.2}), which in turn enables the construction of a global forward attracting set in the weak mean-square topology (Theorem \ref{thm1.3}). Finally, Theorem \ref{thm1.5} asserts that the mean-square semi-dynamical system $\{\Pi_{t,\xi}\}_{t,\xi\in \mathbb{R}^+}$
admits a bounded and closed \(\mathbb{D}\)-attractor \(\mathbb{A}\) on the weighted space \(\mathfrak{C}_{w} \times P\), where $\mathfrak{C}_{w}=C(\mathbb{R}^+, L^{2}_{w}(\Omega, \mathcal{F}; \mathbb{R}^{d}))$. Moreover, the projection of \(\mathbb{A}\) onto \(\mathfrak{C}_{w}\) forms a bounded and closed uniform \(\mathscr{D}\)-attractor \(\mathcal{A}_P\) for the cocycle $\{T_{t,\xi}\}_{t,\xi\in \mathbb{R}^+}$, where \(\mathscr{D}\) denotes the collection of bounded sets in \(\mathfrak{C}_{w}\) and \(\mathbb{D}=\{\mathcal{D}\times P:\mathcal{D}\in \mathscr{D}\)\}.

The rest of this article is organized as follows.
In the next section, we prove that the non-autonomous Caputo FSDE \eqref{eq00.2-2} generates a mean-square semi-dynamical system.
Section 3 is devoted to establishing the uniform boundedness of solutions and
constructing a global forward attracting set in the weak mean-square topology. Section 4 further addresses the existence of mean-square attractors associated with such mean-square semi-dynamical systems.

\noindent{\section{Mean-square semi-dynamical system}}

In this section, we study the existence of the mean-square semi-dynamical system generated by the following  Caputo FSDE:
\begin{align}\label{eq0.1}
D^{\alpha,\varrho} x(t) = g(x(t), \vartheta_t(p))+\frac{dB^{H,\lambda}(t)}{dt},~~~t>0,
\end{align}
with initial values $x_{0}\in L^2(\Omega, \mathcal{F}_{0};\mathbb{R}^d)$ and $p_{0}\in P$.
Here $D^{\alpha,\varrho}$
denotes the Caputo fractional substantial derivative of order $\alpha\in(\frac{1}{2},1)$ with decay parameter $\varrho>0$, $\vartheta_t: P\rightarrow P$, $t\in \mathbb{R}$, form an autonomous dynamical system on the compact base space $P$,
$g: \mathbb{R}^d\times P \to \mathbb{R}^d$ is taken to be Lipschitz continuous (see below for details),
the noise $B^{H,\lambda}(t)$ is a $d$-dimensional TFBM, whose components  $B^{H_{i},\lambda_{i}}_{i}(t)$ are independent one-dimensional TFBMs with Hurst parameters  $H_{i}\in(\frac{1}{2},1)$ and tempering parameters $\lambda_{i}>0$, $i=1,\ldots,d$.

The solution of the Caputo FSDE \eqref{eq0.1} with initial values \(x(0)=x_0 \in L^{2}(\Omega, \mathcal{F}_{0};\mathbb{R}^d)\) and \(p_0 \in P\) satisfies the integral equation
\begin{align}\label{eq0.2}
    x(t, x_0, p_0) = x_0 e^{-\varrho t}+ \int_0^t a(t, s) g(x(s), \vartheta_s(p_0))ds
    +\int_0^t a(t, s)dB^{H,\lambda}(s),
\end{align}
where $a(t, s):=\frac{1}{\Gamma(\alpha)}(t-s)^{\alpha-1}e^{-\varrho (t-s)}$.

We first impose the following assumptions throughout this paper, which can be further relaxed in the subsequent analysis.

\textbf{Assumption 1.} Let \((P, d_P)\) be a compact metric space and let \(\{ \vartheta_t \}_{t \in \mathbb{R}}\) be a group of continuous mappings \(\vartheta_t : P \to P\).

This assumption ensures that the driving system \(\{ \vartheta_t \}_{t \in \mathbb{R}}\) is invertible with
 \(\vartheta_t^{-1} = \vartheta_{-t}\), and that \(P\) is invariant under \(\vartheta_t\), i.e., \(\vartheta_t(P) = P\) for all \(t \in \mathbb{R}\).
Moreover, the compactness of \((P, d_P)\) implies that the metric $d_P$ is bounded; that is, there exists a constant $M>0$ such that $d_P(p, q)\leq  M$ for all $p, q\in P$.

\textbf{Assumption 2.}  There exists a positive constant \( L \) such that for all \( x, y \in L^{2}(\Omega, \mathcal{F};\mathbb{R}^d) \) and \( p, q \in P \),
\[
\mathbb{E}\|g(x, p) - g(y, q)\|^{2} \leq L \mathbb{E}\|x - y\|^{2} + L d^{2}_P(p, q).
\]

Assumptions 1-2 guarantee the existence and uniqueness of solutions  to \eqref{eq0.1}, which depend continuously on the initial value and \(p \in P\) over any finite time interval \([0, T]\). The proof is analogous to that of Theorems 2 and 3 in \cite{Doan-Kloeden} and is omitted here.

For each $t\geq 0$, let $L^{2}(\Omega, \mathcal{F}_t; \mathbb{R}^{d})$ denote the subspace of $L^{2}(\Omega, \mathcal{F}; \mathbb{R}^{d})$ consisting of all $\mathcal{F}_t$-measurable functions.
Let
\[\mathfrak{C} := C(\mathbb{R}^+, L^{2}(\Omega, \mathcal{F}; \mathbb{R}^{d}))\] be the space of continuous functions from $\mathbb{R}^+$ into $L^{2}(\Omega, \mathcal{F}; \mathbb{R}^{d})$.
Correspondingly, for each $t\geq0$, define \[\mathfrak{C}_{t} := C(\mathbb{R}^+, L^{2}(\Omega, \mathcal{F}_{t}; \mathbb{R}^{d})),\]
namely, the space of continuous functions \(f : \mathbb{R}^+ \to L^{2}(\Omega, \mathcal{F}_{t}; \mathbb{R}^{d})\).

For \( p_0 \in P \) and \( f \in \mathfrak{C} \), let $x_{fe^{\varrho\cdot}}(t)$ denote the solution of the following stochastic Volterra integral equation
\begin{align}\label{eq0.3}
x_{fe^{\varrho\cdot}}(t) = f(t) + \int_0^t a(t, s)g(x_{fe^{\varrho\cdot}}(s), \vartheta_s(p_0))ds+\int_0^t a(t, s)dB^{H,\lambda}(s).
\end{align}
Note that \eqref{eq0.3} reduces to the integral equation \eqref{eq0.2} when $f=x_{0}e^{-\varrho\cdot}$.

For each $p\in P$ and $\tau,\xi\geq 0$, an operator
\[T_{\tau,\xi}(\cdot, p):\mathfrak{C}_{\xi}\to \mathfrak{C}_{\xi+\tau}\]
is defined as follows: for $f\in \mathfrak{C}_{\xi}$, the image $(T_{\tau,\xi}(f, p))(\cdot)$ is given, for $\theta\geq0$, by
\begin{align}\label{eq3.1}
(T_{\tau,\xi}(f, p))(\theta)
&:= f(\tau + \theta) + \int_0^\tau a(\tau + \theta, s)g(x_{fe^{\varrho\cdot}}(s), \vartheta_s(p))ds
\nonumber\\
&~~+\int_0^\tau a(\tau + \theta, s)dB^{H,\lambda}(s),
\end{align}
where \( x_{fe^{\varrho\cdot}}(\cdot) \) stands for the solution to the stochastic Volterra integral equation \eqref{eq0.3} associated with \( f \).
In addition, the operator
\[\Pi_{\tau,\xi}:\mathfrak{C}_{\xi} \times P \to \mathfrak{C}_{\tau+\xi} \times P\]
is defined by
\begin{align}\label{eq3.2}
\Pi_{\tau,\xi}(f, p) := (T_{\tau,\xi}(f, p), \vartheta_\tau(p)).
\end{align}

In what follows, we show that $\Pi=\{\Pi_{\tau,\xi}\}_{\tau,\xi\in \mathbb{R}^+}$ forms a semi-group on the space \(\mathfrak{C} \times P\) over $(\Omega, \mathcal{F}, \mathcal{F}_{t}, \mathbf{P})$, which possesses the structure of a skew-product semi-flow.

\begin{theorem}\label{thm1.1}
Let Assumptions 1-2 hold and $f\in \mathfrak{C}_{\xi}$ for some $\xi\geq0$. Then the integral equation \eqref{eq0.3} generates a semi-group of continuous operators $\Pi=\{\Pi_{\tau,\xi}\}_{\tau,\xi\in \mathbb{R}^+}$ on the space \( \mathfrak{C} \times P \) over $(\Omega, \mathcal{F}, \mathcal{F}_{t}, \mathbf{P})$, defined in  \eqref{eq3.2}, which has the structure of a skew-product semi-flow.
\end{theorem}

\begin{proof}
The continuity of the operator $T=\{T_{\tau,\xi}\}_{\tau,\xi\in \mathbb{R}^+}$  follows by an argument analogous to that presented in \cite[Theorem 4]{Doan-Kloeden}, and its proof is therefore omitted. We now proceed to establish that the family $\{T_{\tau,\xi}\}_{\tau,\xi\in \mathbb{R}^+}$ satisfies the cocycle property with respect to the driving system $\{\vartheta_t\}_{t \in \mathbb{R}}$.
In view of \eqref{eq0.3}-\eqref{eq3.1} and $a(t, s):=\frac{1}{\Gamma(\alpha)}(t-s)^{\alpha-1}e^{-\varrho (t-s)}$, we see that
\begin{align}
&x_{fe^{\varrho\cdot}}(t + \tau) \nonumber\\
&= f(t + \tau) + \int_0^{t+\tau} a(t + \tau, s)g(x_{fe^{\varrho\cdot}}(s), \vartheta_s(p))ds+ \int_0^{t+\tau} a(t + \tau, s)dB^{H,\lambda}(s)
\nonumber\\
&= (T_{\tau,\xi}(f, p))(t) +
\int_{\tau}^{t+\tau}
a(t + \tau, s)g(x_{fe^{\varrho\cdot}}(s), \vartheta_s(p))ds
+\int_{\tau}^{t+\tau}
a(t + \tau, s)dB^{H,\lambda}(s)\nonumber\\
&= (T_{\tau,\xi}(f, p))(t) +
\int_0^t a(t, r)g(x_{fe^{\varrho\cdot}}(r + \tau), \vartheta_{r+\tau}(p))
dr+\int_{\tau}^{t+\tau}
a(t + \tau, s)dB^{H,\lambda}(s). \nonumber
\end{align}
By the It\^{o} isometry of TFBM \cite[(2.12)]{Zhang-Wang-Hu}, we have
\begin{align}\label{eq3.3}
&\mathbb{E}\bigg\|
\int_{\tau}^{t+\tau}
a(t + \tau, s)dB^{H,\lambda}(s)\bigg\|^{2}\nonumber\\
&=\sum_{i=1}^{d}\int_{\tau}^{t+\tau}\int_{\tau}^{t+\tau}
a(t + \tau, s)a(t + \tau, r)\phi_{H_{i},\lambda_{i}}(s,r)dsdr
\nonumber\\
&=\sum_{i=1}^{d}\int_{0}^{t}\int_{0}^{t}
a(t, s)a(t, r)\phi_{H_{i},\lambda_{i}}(s,r)dsdr
\nonumber\\
&=\mathbb{E}\bigg\|
\int_{0}^{t}
a(t, s)dB^{H,\lambda}(s)\bigg\|^{2},
\end{align}
where
\begin{align}\label{eq00.400}
\phi_{H_{i},\lambda_{i}}(s,t)
&= \left(H_{i} - \frac{1}{2}\right)^2 \int_0^\infty \tau^{H_{i} - \frac{3}{2}}(|t - s| + \tau)^{H_{i} - \frac{3}{2}}e^{-\lambda_{i}\tau - \lambda_{i}(|t - s| + \tau)}d\tau \nonumber\\
&- \lambda_{i}^2 \int_0^\infty \tau^{H_{i} - \frac{1}{2}}(|t - s| + \tau)^{H_{i} - \frac{1}{2}}e^{-\lambda_{i}\tau - \lambda_{i}(|t - s| + \tau)}d\tau.
\end{align}
This implies that
\begin{align}
x_{fe^{\varrho\cdot}}(t + \tau)
&=
(T_{\tau,\xi}(f, p))(t) +
\int_0^t a(t, r)g(x_{fe^{\varrho\cdot}}(r + \tau), \vartheta_{r+\tau}(p))
dr\nonumber\\
&+\int_{0}^{t}
a(t, s)dB^{H,\lambda}(s)
\end{align}
in $L^{2}(\Omega, \mathcal{F}_{\xi+t + \tau}; \mathbb{R}^{d})$.
From the existence and uniqueness of solutions, it follows that
$x_{fe^{\varrho\cdot}}(t + \tau) = \psi(t)$,
where \(\psi(t)\) is the unique solution of
\[
\psi(t) = (T_{\tau,\xi}(f,p))(t) + \int_0^t a(t,s)g(\psi(s), \vartheta_s(\vartheta_\tau(p)))ds+\int_{0}^{t}
a(t, s)dB^{H,\lambda}(s).
\]
In view of \eqref{eq3.1}, we have
\begin{align}
&(T_{\sigma, \tau+\xi}(T_{\tau, \xi}(f,p)))(\theta) \nonumber\\
&= (T_{\tau, \xi}(f,p))(\sigma + \theta) + \int_0^\sigma a(\sigma + \theta, s)g(\psi(s), \vartheta_s(\vartheta_\tau(p)))ds
+\int_0^\sigma a(\sigma + \theta, s)
dB^{H,\lambda}(s)\nonumber\\
&= f(\tau + \sigma + \theta) + \int_0^\tau a(\tau + \sigma + \theta, s)g(x_{fe^{\varrho\cdot}}(s), \vartheta_s(p))ds
+\int_0^\tau a(\tau + \sigma + \theta, s)
dB^{H,\lambda}(s)\nonumber\\
&+
\int_0^\sigma a(\sigma + \theta, s)g(\psi(s), \vartheta_s(\vartheta_\tau(p)))ds
+\int_0^\sigma a(\sigma + \theta, s)
dB^{H,\lambda}(s)\nonumber\\
&= f(\tau + \sigma + \theta) + \int_0^\tau a(\tau + \sigma + \theta, s)g(x_{fe^{\varrho\cdot}}(s), \vartheta_s(p))ds
+\int_0^\tau a(\tau + \sigma + \theta, s)
dB^{H,\lambda}(s)\nonumber\\
&+
\int_\tau^{\tau + \sigma} a(\sigma +\theta, r - \tau)g(\psi(r - \tau), \vartheta_r(p)) dr
+\int_0^\sigma a(\sigma + \theta, s)
dB^{H,\lambda}(s).\nonumber
\end{align}
Thanks to \(a(\sigma + \theta, r - \tau) = a(\tau + \sigma + \theta, r)\) and \(\psi(r - \tau) = x_{fe^{\varrho\cdot}}(r)\), it holds that
\begin{align}
(T_{\sigma, \tau+\xi}(T_{\tau, \xi}(f,p)))(\theta)
&= f(\tau + \sigma + \theta) + \int_0^{\tau+\sigma} a(\tau + \sigma + \theta, s)g(x_{fe^{\varrho\cdot}}(s), \vartheta_s(p))ds \nonumber\\
&+ \int_0^\tau a(\tau + \sigma + \theta, s)
dB^{H,\lambda}(s)+
\int_0^\sigma a(\sigma + \theta, s)
dB^{H,\lambda}(s).
\end{align}
On the other hand, using \eqref{eq3.1} yields
\begin{align}
(T_{\sigma+\tau,\xi}(f,p))(\theta) &=
f(\tau + \sigma + \theta)+
\int_0^{\tau+\sigma} a(\tau + \sigma + \theta, s)g(x_{fe^{\varrho\cdot}}(s), \vartheta_s(p))ds\nonumber\\
&~~~~+\int_0^{\tau+\sigma} a(\tau + \sigma + \theta, s)dB^{H,\lambda}(s).
\end{align}
By the similar arguments as in \eqref{eq3.3}, we find that for all $\sigma, \tau, \theta\geq 0$ and $f \in \mathfrak{C}_{\xi}$,
\[
(T_{\sigma+\tau, \xi}(f,p))(\theta) = (T_{\sigma, \tau+\xi}(T_{\tau,\xi}(f,p)))(\theta)
\]
in $L^{2}(\Omega, \mathcal{F}_{\xi+\sigma+\tau};\mathbb{R}^{d})$.
Since both sides are continuous in $\theta$, this pointwise identity in $L^{2}(\Omega, \mathcal{F}_{\xi+\sigma+\tau};\mathbb{R}^{d})$ extends to an identity in the space of continuous paths. Hence, the cocycle property
\[
T_{\sigma+\tau, \xi}(f,p) = T_{\sigma, \tau+\xi}(T_{\tau, \xi}(f,p)), \quad \forall \tau, \sigma \geq 0, \quad f \in \mathfrak{C}_{\xi},
\]
holds in \(C(\mathbb{R}^{+},
L^{2}(\Omega, \mathcal{F}_{\xi+\sigma+\tau}; \mathbb{R}^{d}))\).
Moreover, since \(\{ \vartheta_t \}_{t \in \mathbb{R}}\) forms a group of continuous mappings on the compact metric space \((P, d_P)\), the family $\Pi=\{\Pi_{\tau,\xi}\}_{\tau,\xi\in \mathbb{R}^+}$ inherits the analogous cocycle property, namely
\[
\Pi_{\sigma+\tau, \xi} = \Pi_{\sigma, \tau+\xi}\circ \Pi_{\tau, \xi}, \quad \forall \tau, \sigma \geq 0.\]
\end{proof}

In light of \((T_{t,0}(x_0e^{-\varrho\cdot}, p))(0) = x(t, x_0, p)\), the cocycle \(T\) will be used for the Caputo FSDE \eqref{eq0.1} with the specific choice \(f = x_0e^{-\varrho\cdot}\).
For this, we denote \(\mathscr{D}\) by
\begin{align}\label{eq00.400-7}
\mathscr{D} := \left\{ \mathcal{D} \subset \mathfrak{C} : \mathcal{D} = \bigcup_{x \in D} xe^{-\varrho\cdot} \text{ for some bounded set } D \subset L^{2}(\Omega, \mathcal{F}_{0};\mathbb{R}^d) \right\}.
\end{align}
Here a set \( D \) is said to be bounded in $L^{2}(\Omega, \mathcal{F}_{0}; \mathbb{R}^{d})$ if there exists a constant \( R > 0 \) such that \( \mathbb{E}\|x\|^2 < R^2 \) for all \( x \in D \).

\noindent{\section{Global forward attracting set}}

The following theorem establishes the uniform boundedness, with respect to both $p\in P$ and $t\geq0$, of the solution to the Caputo FSDE \eqref{eq0.1}.

\begin{theorem}\label{thm1.2}
Let $\alpha\in(\frac{1}{2},1)$, $H_{i}\in(\frac{1}{2},1),~i=1, \ldots, d$, and Assumptions 1-2 hold.
Then the solution of the Caputo FSDE \eqref{eq0.1} with initial values $x(0)=x_{0}\in L^{2}(\Omega, \mathcal{F}_{0}; \mathbb{R}^{d})$ and $p_{0}\in P$ satisfies
\begin{align}\label{eq3.3-1}
\mathbb{E}\|x(t, x_0, p_{0})\|^{2}
\leq (3e^{-\frac{\varrho}{2} t}\mathbb{E}\|x_0\|^{2}+M_{2})
\bigg[1+\sum_{n=1}^{\infty}\bigg(\frac{12L 2^{\alpha}
}{\varrho^{2\alpha}}\bigg)^{n}
\bigg],
\end{align}
where $M_{2}:=\frac{3}{\varrho^{2\alpha}}
(4L\mathbb{E}\|x_{0}\|^{2}
+2LM^{2}+2\mathbb{E}\|g(x_{0},p_{0})\|^{2})
+\sum_{i=1}^{d}
M_{\varrho, \alpha,H_{i}}$ and $\varrho$ is large enough to ensure the convergence of the above series.
\end{theorem}

\begin{proof}
In view of \eqref{eq0.2} and the It\^{o} isometry of TFBM \cite[(2.12)]{Zhang-Wang-Hu}, we have
\begin{align}\label{eq3.4}
\mathbb{E}\|x(t, x_0, p_{0})\|^{2}
&\leq 3e^{-2\varrho t}\mathbb{E}\|x_0\|^{2}
+3\mathbb{E}\bigg\|
\int_0^t a(t, s) g(x(s), \vartheta_s(p_{0}))ds
\bigg\|^{2}
\nonumber\\
&+3\mathbb{E}\bigg\|
\int_0^t a(t, s)dB^{H,\lambda}(s)\bigg\|^{2}
\nonumber\\
&\leq3e^{-2\varrho t}\mathbb{E}\|x_0\|^{2}
+3\int_0^t a(t, s)ds\int_0^ta(t, s)\mathbb{E}\|
g(x(s), \vartheta_s(p_{0}))\|^{2}ds\nonumber\\
&
+3\sum_{i=1}^{d}\int_0^t\int_0^ta(t, s)a(t, r)\phi_{H_{i},\lambda_{i}}(s,r)dsdr,
\end{align}
where $\phi_{H_i,\lambda_i}$ is given by \eqref{eq00.400}.
According to Assumption 2, we estimate
\begin{align}\label{eq3.5}
&\int_0^ta(t, s)\mathbb{E}\|
g(x(s), \vartheta_s(p_{0}))\|^{2}ds
\nonumber\\
&\leq 2\int_0^ta(t, s)\mathbb{E}\|
g(x(s), \vartheta_s(p_{0}))-g(x_{0},p_{0})\|^{2}ds
+2\int_0^ta(t, s)\mathbb{E}\|g(x_{0},p_{0})\|^{2}ds\nonumber\\
&\leq 2L\int_0^ta(t, s)\mathbb{E}\|x(s)-x_{0}\|^{2}ds
+2L\int_0^ta(t, s)d_{P}^2(\vartheta_s(p_{0}),p_{0})ds
+2\mathbb{E}\|g(x_{0},p_{0})\|^{2}
\int_0^ta(t, s)ds\nonumber\\
&\leq 4L\int_0^ta(t, s)\mathbb{E}\|x(s)\|^{2}ds
+\frac{1}{\varrho^{\alpha}}
(4L\mathbb{E}\|x_{0}\|^{2}
+2LM^{2}+2\mathbb{E}\|g(x_{0},p_{0})\|^{2}).
\end{align}
Notice that $
\phi_{H_{i}, \lambda_{i}}(s, t) \leq (H_i - \frac{1}{2})^2 B(H_i - \frac{1}{2}, 2 - 2H_i) |t-s|^{2H_i-2}$ \cite{Zhang-Wang-Hu}, where $B(H_i - \frac{1}{2}, 2 - 2H_i)$ denotes the Beta function. It follows that
\begin{align}\label{eq3.6}
&\int_0^t\int_0^ta(t, s)a(t, r)\phi_{H_i,\lambda_{i}}(s,r)dsdr\nonumber\\
&\leq \frac{(H_i - \frac{1}{2})^2 B(H_i - \frac{1}{2}, 2 - 2H_i)
}{[\Gamma(\alpha)]^{2}}
\int_0^\infty\int_0^\infty
e^{-\varrho s}e^{-\varrho r}s^{\alpha-1}r^{\alpha-1}
|s-r|^{2H_i-2}dsdr\nonumber\\
&=\frac{2(H_i - \frac{1}{2})^2 B(H_i - \frac{1}{2}, 2 - 2H_i)
}{[\Gamma(\alpha)]^{2}}
\int_0^\infty e^{-\varrho s}s^{\alpha-1}
\int_0^s e^{-\varrho r}
r^{\alpha-1}
(s-r)^{2H_i-2}drds\nonumber\\
&=\frac{2(H_i - \frac{1}{2})^2 B(H_i - \frac{1}{2}, 2 - 2H_i)
}{[\Gamma(\alpha)]^{2}}
\int_0^\infty e^{-\varrho s}s^{2\alpha+2H_i-3}
\int_0^1 e^{-\varrho su}u^{\alpha-1}(1-u)^{2H_i-2}duds\nonumber\\
&=\frac{2(H_i - \frac{1}{2})^2 B(H_i - \frac{1}{2}, 2 - 2H_i)
}{[\Gamma(\alpha)]^{2}}
\int_0^1u^{\alpha-1}(1-u)^{2H_i-2}
\int_0^\infty e^{-\varrho s(1+u)}s^{2\alpha+2H_i-3}dsdu\nonumber\\
&=\frac{2(H_i - \frac{1}{2})^2 B(H_i - \frac{1}{2}, 2 - 2H_i)
\Gamma(2\alpha+2H_i-2)}{[\Gamma(\alpha)]^{2}\varrho^{2\alpha+2H_i-2}}
\int_0^1u^{\alpha-1}(1-u)^{2H_i-2}(1+u)^{-(2\alpha+2H_i-2)}du\nonumber\\
&\leq \frac{2(H_i - \frac{1}{2})^2 B(H_i - \frac{1}{2}, 2 - 2H_i)
\Gamma(2\alpha+2H_i-2)B(\alpha,2H_i-1)}
{[\Gamma(\alpha)]^{2}\varrho^{2\alpha+2H_i-2}}
:=\frac{M_{\varrho, \alpha,H_i}}{3}.
\end{align}
Inserting \eqref{eq3.5}-\eqref{eq3.6} into \eqref{eq3.4} results in
\begin{align}
\mathbb{E}\|x(t, x_0, p_{0})\|^{2}
&\leq3e^{-2\varrho t}\mathbb{E}\|x_0\|^{2}
+\frac{3}{\varrho^{2\alpha}}
(4L\mathbb{E}\|x_{0}\|^{2}
+2LM^{2}+2\mathbb{E}\|g(x_{0},p_{0})\|^{2})\nonumber\\
&
+\sum_{i=1}^{d}
M_{\varrho, \alpha,H_{i}}+\frac{12L}{\varrho^{\alpha}\Gamma(\alpha)}
\int_0^t
(t-s)^{\alpha-1}e^{-\varrho (t-s)}
\mathbb{E}\|x(s)\|^{2}ds.\nonumber
\end{align}
Set $M_{2}:=\frac{3}{\varrho^{2\alpha}}
(4L\mathbb{E}\|x_{0}\|^{2}
+2LM^{2}+2\mathbb{E}\|g(x_{0},p_{0})\|^{2})
+\sum_{i=1}^{d}
M_{\varrho, \alpha,H_{i}}$. Then the above expression can be rewritten as
\begin{align}
e^{\varrho t}\mathbb{E}\|x(t, x_0, p_{0})\|^{2}
&\leq3e^{-\varrho t}\mathbb{E}\|x_0\|^{2}
+\frac{12L}{\varrho^{\alpha}\Gamma(\alpha)}
\int_0^t
(t-s)^{\alpha-1}e^{\varrho s}
\mathbb{E}\|x(s)\|^{2}ds+M_{2}e^{\varrho t}.\nonumber
\end{align}
From Corollary 2.3 in \cite{Wang-Xu-Kloeden}, we obtain that
\begin{align}
e^{\varrho t}\mathbb{E}\|x(t, x_0, p_{0})\|^{2}
\leq (3e^{\frac{\varrho}{2} t}\mathbb{E}\|x_0\|^{2}+M_{2}e^{\varrho t})
\bigg[1+\sum_{n=1}^{\infty}
\bigg(\frac{12L 2^{\alpha}}{\varrho^{2\alpha}}\bigg)^{n}
\bigg],\nonumber
\end{align}
which completes the theorem.
\end{proof}

It then follows that, for sufficiently large $\varrho$, the bounded set $\mathcal{B}^*$, defined by
\begin{align}\label{eq3.7}
\mathcal{B}^* := \left\{ x \in L^{2}(\Omega, \mathcal{F};\mathbb{R}^d) : \mathbb{E}\|x(t,x_{0},p_{0})\|^2 \leq 2+2M_{2}:= R_*^2 \right\}
\end{align}
uniformly absorbs, with respect to \( p_{0} \in P \), the solutions of the Caputo FSDE \eqref{eq0.1}. More precisely, there exist constants $\tilde{\varrho}>0$ and  $\tilde{t}_{\tilde{\varrho}}> 0$,  independent of \( p_0 \in P \), such that for all initial data $\mathbb{E}\|x_0\|^2 \le M$ (with $M>0$), $\varrho\geq \tilde{\varrho}$, \( t \geq  \tilde{t}_{\tilde{\varrho},M}\) and \( p_0 \in P \), the inequality
\[ \mathbb{E}\|x(t, x_0, p_0)\|^{2}  \leq R_*^2\]
holds.

Let $L^{2}_{w}(\Omega, \mathcal{F}; \mathbb{R}^d)$ denote the space  $L^{2}(\Omega, \mathcal{F}; \mathbb{R}^d)$ endowed with the weak topology.
The existence of a global forward attracting set in this weak topology for the Caputo FSDE \eqref{eq0.1} is then established in the following theorem.

\begin{theorem}\label{thm1.3}
Let $\alpha\in(\frac{1}{2},1)$, $H_{i}\in(\frac{1}{2},1),~i=1, \ldots, d$, and Assumptions 1-2 hold.
Suppose that $\varrho$ is sufficiently large. Then
\begin{itemize}
\item [(i)] for any bounded subset \( D \) of \(L^2(\Omega, \mathcal{F}_{0}; \mathbb{R}^d)\), any sequence \(\{t_n\}\) with \(t_n \to +\infty\) as \(n \to +\infty\), \(\{x_{0,n}\}_{n \in \mathbb{N}} \subset D\), \(\{p_n\}_{n \in \mathbb{N}} \subset P\),
    and any sequence of solutions \(\{x(\cdot, x_{0,n}, p_n)\}_{n \in \mathbb{N}} \) of the Caputo FSDE \eqref{eq0.1}, the sequence \(\{x(t_n, x_{0,n}, p_n)\}_{n \in \mathbb{N}} \) is relatively compact in \(L^2_{w}(\Omega, \mathcal{F}; \mathbb{R}^d)\);
  \item [(ii)]for any bounded subset \(D\) of \(L^2(\Omega, \mathcal{F}_{0}; \mathbb{R}^d)\), the set
\begin{align}
\omega_p(D) =
&\bigg\{ y : \exists~t_n \to +\infty, \{x_{0,n}\}_{n \in \mathbb{N}} \subset D, \{p_n\}_{n \in \mathbb{N}} \subset P,~\mbox{and a}\nonumber\\
&~\mbox{sequence of~solutions }\{x(\cdot, x_{0,n}, p_n)\}_{n \in \mathbb{N}}~\mbox{of~the~Caputo}\nonumber\\
&~~\mbox{FSDE~\eqref{eq0.1} such~that}~x(t_n, x_{0,n}, p_n) \to y~\mbox{in}~L^2_{w}(\Omega, \mathcal{F}; \mathbb{R}^d)\bigg\}\nonumber
\end{align}

is nonempty, compact and attracts \(D\) in the weak topology;
  \item [(iii)] the set
\[
\Omega_p^* = \overline{\bigcup\{\omega_p(D) : D \subset L^2(\Omega, \mathcal{F}_{0}; \mathbb{R}^d), \ D \text{ bounded}\}}^{w}
\]

is bounded in \(L^2(\Omega, \mathcal{F}; \mathbb{R}^d)\), compact in the topology of \(L^2_{w}(\Omega, \mathcal{F}; \mathbb{R}^d)\), and, moreover, is the minimal weakly closed set that attracts all bounded subsets of \(L^2(\Omega, \mathcal{F}_{0}; \mathbb{R}^d)\) in the weak topology.
\end{itemize}

\end{theorem}

\begin{proof}
By the reflexivity of \(L^2(\Omega, \mathcal{F}; \mathbb{R}^d)\), conclusion $(i)$ is a direct consequence of Theorem \ref{thm1.2}.
It follows that \(\omega_p(D)\) is nonempty and weakly compact.

To show that \(\omega_p(D)\) attracts \(D\) in the weak topology, we argue by contradiction.
Suppose that there exist \(\varepsilon_0 > 0\) and sequences \(\{t_n\}\) with \(t_n \to +\infty\) as \(n \to +\infty\), \(\{x_{0,n}\}_{n \in \mathbb{N}}\subset D\),\(\{p_n\}_{n \in \mathbb{N}} \subset P\), and solutions \(x(\cdot, x_{0,n}, p_n)\) of \eqref{eq0.1} such that
\begin{align}\label{eq3.8}
\operatorname{dist}_w(x(t_n, x_{0,n}, p_n), \omega_p(D)) > \varepsilon_0, \quad \forall n \in \mathbb{N},
\end{align}
where \(\operatorname{dist}_w(\cdot, \cdot)\) is in the sense of weak topology. By conclusion $(i)$, \(x(t_n, x_{0,n}, p_n)\) is relatively compact in \(L^2_{w}(\Omega, \mathcal{F}; \mathbb{R}^d)\) and possesses at least one cluster point \(w\). According to the definition of \(\omega_p(D)\), we have \(w \in \omega_p(D)\), which contradicts \eqref{eq3.8}.

By Theorem \ref{thm1.2} and the conclusion $(ii)$, \( \Omega_p^* \) is bounded in \( L^2(\Omega, \mathcal{F}; \mathbb{R}^d) \) and attracts all bounded subsets of \( L^2(\Omega, \mathcal{F}_{0}; \mathbb{R}^d) \) in the weak topology. Moreover, it is evidently that \( \Omega_p^*\) is compact in the topology of \( L^2_{w}(\Omega, \mathcal{F}; \mathbb{R}^d) \).

It remains to show that \( \Omega_p^* \) is the minimal weakly closed set attracting any bounded set \( D \subset L^2(\Omega, \mathcal{F}_{0}; \mathbb{R}^d) \) in the weak topology. Indeed,
suppose there exists another weakly closed set
\( \Omega_p^{*'} \) that also attracts any bounded set \( D \subset L^2(\Omega, \mathcal{F}_{0}; \mathbb{R}^d) \) in the weak topology, then the definition of \( \omega_p(D) \) directly yields
\( \omega_p(D) \subset \Omega_p^{*'} \),
and thus
\[
\bigcup\{\omega_p(D) : D \subset L^2(\Omega, \mathcal{F}_{0}; \mathbb{R}^d),D \text{ bounded}\} \subset \Omega_p^{*'}.
\]
Since \( \Omega_p^{*'} \) is weakly closed, we have
\[
\Omega_p^* = \overline{\bigcup\{\omega_p(D) : D \subset L^2(\Omega, \mathcal{F}_{0}; \mathbb{R}^d), \ D \text{ bounded}\}}^{w} \subseteq \Omega_p^{*'}.
\]
This completes the proof of Theorem \ref{thm1.3}.
\end{proof}

\noindent{\section{Mean-square attractors}}
We now turn to exploring the existence of attractors for the cocycle \( T \) as well as the skew-product semi-flow \( \Pi \) defined by \eqref{eq3.1} and \eqref{eq3.2}, respectively.

Recall that \(L^2_{w}(\Omega, \mathcal{F}; \mathbb{R}^d)\) is the space \(L^2(\Omega, \mathcal{F}; \mathbb{R}^d)\) equipped with the weak topology. Define
\[
\mathfrak{C}_{t,w}:=C([0,\infty),
L^{2}_{w}(\Omega, \mathcal{F}_{t}; \mathbb{R}^d)),
~~~~
\mathfrak{C}_{w}:=C([0,\infty),
L^{2}_{w}(\Omega, \mathcal{F}; \mathbb{R}^d)).\]
For $\alpha\in (\frac{1}{2},1)$, we equip $\mathfrak{C}_{w}$ with the norm
\begin{align}\label{eq3.8-7-8}
\|f\|^{2}_\alpha := \mathbb{E}\|f(0)\|^{2} + \sum_{N=1}^\infty \frac{1}{2^N N^{2\alpha}} \mathbb{E}\|f\|^{2}_N,
\end{align}
with
\[
\mathbb{E}\|f\|^{2}_N := \sup_{t \in [0, N]} \mathbb{E}\|f(t)\|^{2}, \quad N = 1, 2, \dots.
\]

It follows that \((\mathfrak{C}_{w}, \|\cdot\|_\alpha)\) is a complete Banach space. With the family   \(\mathscr{D}\subset \mathfrak{C}_w\) defined in \eqref{eq00.400-7}, we now define the mean-square uniform (with respect to \(p \in P\)) \(\mathscr{D}\)-attractor of the cocycle \(T\) in \(\mathfrak{C}_{w}\) as follows.

\begin{definition}\label{def1.4}
A nonempty set \(\mathcal{A}_P\) in \(\mathfrak{C}_{w}\) is called the mean-square uniform (with respect to \(p \in P\)) \(\mathscr{D}\)-attractor of the cocycle \(T\) if
    \begin{enumerate}
        \item[(i)] \(\mathcal{A}_P\) is closed and bounded in \(\mathfrak{C}_{w}\);
        \item[(ii)] \(\mathcal{A}_P\) is uniformly (with respect to \(p \in P\)) \(\mathscr{D}\)-attracting under \(T\), that is, for any \(\mathcal{D} \in \mathscr{D}\),
        \[
        \lim_{t \to \infty} \operatorname{dist}_{\mathfrak{C}_{w}}
        \bigl(T_{t,0}(\mathcal{D}, P), \mathcal{A}_P\bigr) = 0;
        \]
        \item[(iii)] Minimality: If \(\mathcal{A}'_P\) is a closed set in \(\mathfrak{C}_{w}\) also uniformly \(\mathscr{D}\)-attracting, then \(\mathcal{A}_P \subset \mathcal{A}'_P\).
    \end{enumerate}
\end{definition}

Define the universe \(\mathbb{D}\) of some kind of bounded sets in \(\mathfrak{C}_{w} \times P\) by
\[
\mathbb{D} = \{ \mathcal{D} \times P : \mathcal{D} \in \mathscr{D} \}.
\]
On the product space \(\mathbb{X} := \mathfrak{C}_w \times P\) metrized by \(d_{\mathbb{X}} = d_{\mathfrak{C}_w} + d_P\), the family  \(\Pi = \{\Pi_{t,\xi}\}_{t,\xi \in \mathbb{R}^+}\) defined in \eqref{eq3.2} forms a semi-group.
Analogously, we give the definition of the mean-square \(\mathbb{D}\)-attractor for the semi-group \(\Pi\) in \(\mathfrak{C}_w \times P\) as follows.

\begin{definition}\label{def1.5}
    A nonempty set \(\mathbb{A}\) in \(\mathbb{X} = \mathfrak{C}_{w} \times P\) is called the mean-square \(\mathbb{D}\)-attractor of the semi-group \(\Pi\) if
    \begin{enumerate}
        \item[(i)] \(\mathbb{A}\) is closed and bounded in \(\mathbb{X}\);
        \item[(ii)] \(\mathbb{A}\) is \(\mathbb{D}\)-attracting under \(\Pi\), that is, for any \(\mathcal{D} \times P \in \mathbb{D}\),
        \[
        \lim_{t \to \infty} \operatorname{dist}_{\mathbb{X}}\bigl(\Pi_{t,0}(\mathcal{D} \times P), \mathbb{A}\bigr) = 0;
        \]
        \item[(iii)]  Minimality: If \(\mathbb{A}'\) is a closed \(\mathbb{D}\)-attracting set in \(\mathbb{X}\), then \(\mathbb{A} \subset \mathbb{A}'\).
    \end{enumerate}
\end{definition}

To establish the existence of mean-square attractors for the Caputo FSDE \eqref{eq0.1}, it is essential first to obtain various estimates on its solutions in \( L^{2}(\Omega, \mathcal{F};\mathbb{R}^d) \).

\subsection{Boundedness and continuity of solution}

Under Assumptions 1-2, let \( x(t, x_0, p_0) \) denote the solution to the Caputo FSDE \eqref{eq0.1} with initial values \( x(0, x_0, p_0) = x_0 \) and $p_0\in P$. This solution satisfies estimate \eqref{eq3.3-1}, is absorbed into the set \( \mathcal{B}^* \subset L^{2}(\Omega, \mathcal{F};\mathbb{R}^d) \) given in \eqref{eq3.7} within finite time, and remains in \( \mathcal{B}^*\) thereafter whenever $\varrho$ is sufficiently large. The estimates below hold true.
\begin{align}\label{eq3.9}
B_R := \sup_{t \geq 0, \mathbb{E}\|x_0\|^{2} \leq R^{2}, p_{0} \in P} \mathbb{E}\|x(t, x_0, p_0)\|^{2} < \infty,
\end{align}
\begin{align}\label{eq3.10}
B_R^g := \sup_{\mathbb{E}\|x\|^{2} \leq B_R, p \in P} \mathbb{E}\|g(x, p)\|^{2} < \infty,
\end{align}
where Assumption 2 has been used in the second bound. Notice that \eqref{eq3.9} and \eqref{eq3.10} are valid for \( R = R_* \) provided \( t \geq  \tilde{t}_{\varrho, R}\).

\begin{lemma}\label{lem1.6}
Let $\alpha\in (\frac{1}{2},1)$, $H_{i}\in(\frac{1}{2},1),~i=1, \ldots, d$, and Assumptions 1-2 hold.  Then the solution of the Caputo FSDE \eqref{eq0.1} is continuous in time. In particular, there exists a constant $c>0$ such that
\begin{align}
\mathbb{E}\|x(t + \theta, x_0, p_0) - x(t, x_0, p_0)\|^{2}
\leq c(\theta^{2\alpha}+\theta^{2}
+\sum_{i=1}^{d}\theta^{2H_i+2\alpha-2})
\nonumber
\end{align}
for any $\theta\geq0$, $t\geq0$,  \(\mathbb{E}\|x_0\|^{2} \leq R^{2}\) and \(p_0 \in P\), where $c$ may depend on $R, \alpha, \varrho, H, B_{R}^{g},d$.

\end{lemma}

\begin{proof}
In view of \eqref{eq0.2}, we have
\begin{align}
&\mathbb{E}\|x(t + \theta, x_0, p_0) - x(t, x_0, p_0)\|^{2}
\leq 3\mathbb{E}\|x_0\|^{2} e^{-2\varrho t}(1-e^{-\varrho \theta})^{2}\nonumber\\
&+ 3\mathbb{E}\bigg\|\int_0^{t+\theta}
a(t+\theta, s) g(x(s), \vartheta_s(p_0))ds
-
\int_0^{t}
a(t, s) g(x(s), \vartheta_s(p_0))ds\bigg\|^{2}
\nonumber\\
&+3
\mathbb{E}\bigg\|\int_0^{t+\theta}
a(t+\theta, s)dB^{H,\lambda}(s)-
\int_0^{t}
a(t, s)dB^{H,\lambda}(s)\bigg\|^{2}
\nonumber\\
&:=3J_{1}+3J_{2}+3J_{3}.\nonumber
\end{align}
It is easy to see that
\begin{align}
J_{1}\leq \mathbb{E}\|x_0\|^{2} e^{-2\varrho t}\varrho^{2}\theta^{2}.\nonumber
\end{align}
Taking into account \eqref{eq3.10}, we deduce that
\begin{align}
J_{2}&\leq
2\mathbb{E}\bigg\|\int_t^{t+\theta}
a(t+\theta, s) g(x(s), \vartheta_s(p_0))ds\bigg\|^{2}
\nonumber\\
&+2\mathbb{E}\bigg\|\int_0^{t}
[a(t+\theta, s)-a(t, s)]g(x(s), \vartheta_s(p_0))ds\bigg\|^{2}
\nonumber\\
&\leq 2\int_t^{t+\theta}\int_t^{t+\theta}
a(t+\theta, s)a(t+\theta, r)
\mathbb{E}\|g(x(s), \vartheta_s(p_0))\|^{2}
dsdr
\nonumber\\
&+2\int_0^{t}\int_0^{t}
[a(t+\theta, s)-a(t, s)]
[a(t+\theta, r)-a(t, r)]
\mathbb{E}\|g(x(s), \vartheta_s(p_0))\|^{2}
dsdr\nonumber\\
&\leq 2B_R^g\int_t^{t+\theta}\int_t^{t+\theta}
a(t+\theta, s)a(t+\theta, r)
dsdr
\nonumber\\
&+2B_R^g\int_0^{t}\int_0^{t}
[a(t+\theta, s)-a(t, s)]
[a(t+\theta, r)-a(t, r)]
dsdr\nonumber\\
&\leq \frac{2B_R^g \theta^{2\alpha}}{[\Gamma(\alpha+1)]^{2}}
+\frac{16B_R^g\theta^{2\alpha}}{[\Gamma(\alpha)]^{2}}
+\frac{16B_R^g\theta^{2\alpha}}
{[\Gamma(\alpha+1)]^{2}}\nonumber\\
&=\frac{18B_R^g \theta^{2\alpha}}{[\Gamma(\alpha+1)]^{2}}
+\frac{16B_R^g\theta^{2\alpha}}{[\Gamma(\alpha)]^{2}},\nonumber
\end{align}
where we have used $\int_t^{t+\theta}
a(t+\theta, s)ds\leq \frac{\theta^{\alpha}}{\Gamma(\alpha+1)}$ and
\begin{align}
&\int_0^{t}
[a(t, s)-a(t+\theta, s)]ds\nonumber\\
&=\frac{1}{\Gamma(\alpha)}
\int_{0}^{t}
(t+\theta-s)^{\alpha-1}
\bigg(e^{-\varrho (t-s)}-e^{-\varrho (t+\theta-s)}
\bigg)ds\nonumber\\
&+\frac{1}{\Gamma(\alpha)}
\int_{0}^{t}e^{-\varrho (t-s)}
[(t-s)^{\alpha-1}-(t+\theta-s)^{\alpha-1}]ds
\nonumber\\
&\leq
\frac{1}{\Gamma(\alpha)}
(1-e^{-\varrho \theta})
\int_{0}^{t}
(r+\theta)^{\alpha-1}
e^{-\varrho r}dr
+\frac{1}{\Gamma(\alpha)}
\int_{0}^{t}
[r^{\alpha-1}-(r+\theta)^{\alpha-1}]dr\nonumber\\
&\leq \frac{2\theta^{\alpha}}{\Gamma(\alpha)}
+\frac{[t^{\alpha}+\theta^{\alpha}-(t+\theta)^{\alpha}]}
{\Gamma(\alpha+1)}\nonumber\\
&\leq \frac{2\theta^{\alpha}}{\Gamma(\alpha)}
+\frac{2\theta^{\alpha}}{\Gamma(\alpha+1)}.\nonumber
\end{align}
By the It\^{o} isometry of TFBM \cite[(2.12)]{Zhang-Wang-Hu}, we estimate
\begin{align}
J_{3}&\leq
\frac{3}{[\Gamma(\alpha)]^{2}}\mathbb{E}\bigg\|
\int_{t}^{t+\theta}
(t+\theta-s)^{\alpha-1}e^{-\varrho (t+\theta-s)}
dB^{H,\lambda}(s)\bigg\|^{2}\nonumber\\
&+\frac{3}{[\Gamma(\alpha)]^{2}}\mathbb{E}\bigg\|
\int_{0}^{t}[(t-s)^{\alpha-1}-(t+\theta-s)^{\alpha-1}]
e^{-\varrho (t+\theta-s)}
dB^{H,\lambda}(s)\bigg\|^{2}\nonumber\\
&+\frac{3}{[\Gamma(\alpha)]^{2}}\mathbb{E}\bigg\|
\int_{0}^{t}(t-s)^{\alpha-1}
[e^{-\varrho (t+\theta-s)}-e^{-\varrho (t-s)}]
dB^{H,\lambda}(s)\bigg\|^{2}\nonumber\\
&\leq \frac{3}{[\Gamma(\alpha)]^{2}}
\sum_{i=1}^{d}
\int_{t}^{t+\theta}
\int_{t}^{t+\theta}
(t+\theta-s)^{\alpha-1}
(t+\theta-r)^{\alpha-1}
e^{-\varrho (t+\theta-s)}
e^{-\varrho (t+\theta-r)}
\phi_{H_{i},\lambda_{i}}(s,r)dsdr
\nonumber\\
&+\frac{3}{[\Gamma(\alpha)]^{2}}
\sum_{i=1}^{d}
\int_{0}^{t}\int_{0}^{t}
[(t-s)^{\alpha-1}-(t+\theta-s)^{\alpha-1}]
[(t-r)^{\alpha-1}-(t+\theta-r)^{\alpha-1}]
\nonumber\\
&~~\times
e^{-\varrho(t-s+\theta)}e^{-\varrho(t-r+\theta)}
\phi_{H_{i},\lambda_{i}}(s,r)dsdr\nonumber\\
&+\frac{3}{[\Gamma(\alpha)]^{2}}
\sum_{i=1}^{d}
\int_{0}^{t}\int_{0}^{t}
[e^{-\varrho (t+\theta-s)}-e^{-\varrho (t-s)}]
[e^{-\varrho (t+\theta-r)}-e^{-\varrho (t-r)}]
\nonumber\\
&~~\times
(t-s)^{\alpha-1}(t-r)^{\alpha-1}
\phi_{H_{i},\lambda_{i}}(s,r)
dsdr\nonumber\\
&\leq \frac{c_{H}}{[\Gamma(\alpha)]^{2}}
\sum_{i=1}^{d}
\int_{0}^{\theta}
\int_{0}^{\theta}
s^{\alpha-1}
r^{\alpha-1}
|r-s|^{2H_{i}-2}dsdr
\nonumber\\
&+
\frac{c_{H}}{[\Gamma(\alpha)]^{2}}
\sum_{i=1}^{d}
\int_{0}^{t}\int_{0}^{t}
[s^{\alpha-1}-(s+\theta)^{\alpha-1}]
[r^{\alpha-1}-(r+\theta)^{\alpha-1}]
e^{-\varrho(s+\theta)}e^{-\varrho(r+\theta)}
|r-s|^{2H_{i}-2}dsdr
\nonumber\\
&+
\frac{c_{H}}{[\Gamma(\alpha)]^{2}}[1-e^{-\varrho \theta}]^{2}
\sum_{i=1}^{d}
\int_{0}^{t}\int_{0}^{t}
s^{\alpha-1}r^{\alpha-1}
e^{-\varrho s}e^{-\varrho r}
|r-s|^{2H_{i}-2}dsdr\nonumber\\
&\leq
c_{H,\alpha, \varrho}
\sum_{i=1}^{d}\theta^{2H_i+2\alpha-2}
+c_{H,\alpha, \varrho, d}\theta^{2},\nonumber
\end{align}
where we have used
\begin{align}\label{eq3.10-01}
&\int_{0}^{\theta}
\int_{0}^{\theta}
s^{\alpha-1}
r^{\alpha-1}
|r-s|^{2H_{i}-2}dsdr\nonumber\\
&=2\int_{0}^{\theta}
\int_{0}^{r}
s^{\alpha-1}
r^{\alpha-1}
(r-s)^{2H_{i}-2}dsdr\nonumber\\
&=\frac{2B(\alpha, 2H_{i}-1)\theta^{2\alpha+2H_{i}-2}}{2\alpha+2H_{i}-2},
\end{align}
and
\begin{align}
&\int_{0}^{\infty}\int_{0}^{\infty}
[s^{\alpha-1}-(s+\theta)^{\alpha-1}]
[r^{\alpha-1}-(r+\theta)^{\alpha-1}]
e^{-\varrho s}e^{-\varrho r}
|r-s|^{2H_{i}-2}dsdr\nonumber\\
&\leq \theta^{2H_{i}+2\alpha-2}
\int_{0}^{\infty}\int_{0}^{\infty}
[s^{\alpha-1}-(s+1)^{\alpha-1}]
[r^{\alpha-1}-(r+1)^{\alpha-1}]
e^{-\theta\varrho s}e^{-\theta\varrho r}
|r-s|^{2H_{i}-2}dsdr\nonumber\\
&\leq c_{H_{i},\alpha, \varrho}\theta^{2H_{i}+2\alpha-2}.\nonumber
\end{align}
Therefore, we conclude that there exists a constant $c>0$ such that
\begin{align}
\mathbb{E}\|x(t + \theta, x_0, p_0) - x(t, x_0, p_0)\|^{2}
\leq c(\theta^{2\alpha}+\theta^{2}
+\sum_{i=1}^{d}\theta^{2H_i+2\alpha-2}
),
\nonumber
\end{align}
where $c$ may depend on $R, \alpha, \varrho, H, B_{R}^{g}$ and $d$.
\end{proof}

In addition, we obtain the following estimate for  \(T_{t,0}(x_0e^{-\varrho\cdot}, p_{0})\).

\begin{lemma}\label{lem1.7}
For \(\mathbb{E}\|x_0\|^{2} \leq R^{2}\), \( p_0 \in P \), $\varrho\geq \tilde{\varrho}$, \( t \geq  \tilde{t}_{\tilde{\varrho}, R}\) and $\theta\geq0$, we have
\begin{align}\label{eq3.11}
&\mathbb{E}
\|(T_{t,0}(x_0e^{-\varrho\cdot}, p_{0}))(\theta)\|^{2}
\nonumber\\
&\leq \frac{4B_{R_*}^g \theta^{2\alpha}}{[\Gamma(\alpha+1)]^{2}}
+2\sum_{i=1}^{d}M_{3,i}\theta^{2\alpha+2H_{i}-2}
+2R_*^{2},
\end{align}
where $M_{3,i}=\frac{4(H_{i}- \frac{1}{2})^2 B(H_{i}- \frac{1}{2}, 2 - 2H_{i})B(\alpha, 2H_{i}-1)}{[\Gamma(\alpha)]^{2}(2\alpha+2H_{i}-2)}$, $i=1, \ldots, d$, and \( R_* = \sqrt{2+2M_{2}} \) is given in \eqref{eq3.7}.
\end{lemma}

\begin{proof}
In view of \eqref{eq0.2} and \eqref{eq3.1}, we find that
\begin{align}
&x(t+\theta, x_0, p_0) \nonumber\\
&= x_0 e^{-\varrho (t+\theta)}+
\int_0^{t+\theta} a(t+\theta, s) g(x(s), \vartheta_s(p_0))ds
+\int_0^{t+\theta} a(t+\theta, s)dB^{H,\lambda}(s)
\nonumber\\
&=(T_{t,0}(x_0 e^{-\varrho\cdot}, p_0))(\theta)
+\int_{t}^{t+\theta} a(t+\theta, s) g(x(s), \vartheta_s(p_0))ds
+\int_{t}^{t+\theta}a(t+\theta, s)dB^{H,\lambda}(s).\nonumber
\end{align}
Recall that \( x(t, x_0, p_0) \in \mathcal{B}^*, \) i.e.,
\(\mathbb{E}\|x(t, x_0, p_0)\|^{2} \leq R_*^{2}\) for all \(\mathbb{E}\|x_0\|^{2} \leq R^{2}\), \( p_0 \in P \), $\varrho\geq \tilde{\varrho}$ and \( t \geq  \tilde{t}_{\tilde{\varrho}, R}\).
Hence, by \eqref{eq3.10-01} we obtain that
\begin{align}\label{eq3.11-1}
&\mathbb{E}\|x(t + \theta, x_0, p_0) - (T_{t,0}(x_0 e^{-\varrho\cdot}, p_0))(\theta)\|^{2} \nonumber\\
&\leq  2\mathbb{E}\bigg\|
\int_{t}^{t+\theta} a(t+\theta, s) g(x(s), \vartheta_s(p_0))ds
\bigg\|^{2} +2\mathbb{E}\bigg\|
\int_{t}^{t+\theta}a(t+\theta, s)dB^{H,\lambda}(s)
\bigg\|^{2} \nonumber\\
&\leq  \frac{2}{[\Gamma(\alpha)]^{2}}
\int_{t}^{t+\theta}\int_{t}^{t+\theta}
(t+\theta-s)^{\alpha-1}(t+\theta-r)^{\alpha-1}
\mathbb{E}\|g(x(s), \vartheta_s(p_0))\|^{2}dsdr
\nonumber\\
&+\sum_{i=1}^{d}
\frac{2(H_{i} - \frac{1}{2})^2 B(H_{i} - \frac{1}{2}, 2 - 2H_{i})}{[\Gamma(\alpha)]^{2}}
\int_{t}^{t+\theta}\int_{t}^{t+\theta}
(t+\theta-s)^{\alpha-1}(t+\theta-r)^{\alpha-1}
|s-r|^{2H_{i}-2}dsdr
\nonumber\\
&\leq \frac{2B_{R_*}^g \theta^{2\alpha}}{[\Gamma(\alpha+1)]^{2}}
+\sum_{i=1}^{d}M_{3,i}\theta^{2\alpha+2H_{i}-2}
\end{align}
for all \(\mathbb{E}\|x_0\|^{2} \leq R^{2}\), \( p_0 \in P \), $\varrho\geq \tilde{\varrho}$ and \( t \geq  \tilde{t}_{\tilde{\varrho}, R}\), where $M_{3,i}=\frac{4(H_{i}- \frac{1}{2})^2 B(H_{i}- \frac{1}{2}, 2 - 2H_{i})B(\alpha, 2H_{i}-1)}{[\Gamma(\alpha)]^{2}(2\alpha+2H_{i}-2)}$. Moreover,
\begin{align}
&\mathbb{E}\|(T_{t,0}(x_0 e^{-\varrho\cdot}, p_{0}))(\theta)
\|^{2}\nonumber\\
&\leq 2\mathbb{E}\|x(t + \theta, x_0, p_0) - (T_{t,0}(x_0 e^{-\varrho\cdot}, p_{0}))(\theta)\|^{2}
+2\mathbb{E}\|x(t + \theta, x_0, p_0)\|^{2}\nonumber\\
&\leq \frac{4B_{R_*}^g \theta^{2\alpha}}{[\Gamma(\alpha+1)]^{2}}
+2\sum_{i=1}^{d}M_{3,i}\theta^{2\alpha+2H_{i}-2}
+2R_*^{2}\nonumber
\end{align}
for all \(\mathbb{E}\|x_0\|^{2} \leq R^{2}\), \( p_0 \in P \), $\varrho\geq \tilde{\varrho}$ and \( t \geq  \tilde{t}_{\tilde{\varrho}, R}\).
\end{proof}

\subsection{A uniformly $\mathscr{D}$-absorbing set of $T$ in $\mathfrak{C}_w$}

Consider an arbitrary element $\mathcal{D} \in \mathscr{D}$. Thanks to the definition of $\mathscr{D}$, there exists a bounded subset $D\subset L^{2}(\Omega, \mathcal{F}_{0};\mathbb{R}^d)$ such that $\mathcal{D} = \{xe^{-\varrho\cdot}: x \in D\}$.
Without loss of generality, we may assume that $\sup_{x \in D} \mathbb{E}\|x\|^{2} \leq R^{2}$, with $R \geq R_*$.
For any given $x_0 \in D$ and $p_0 \in P$, we define the function  $\chi(t, \theta) := (T_{t,0}(x_0 e^{-\varrho\cdot}, p_0))(\theta)$, which satisfies $\chi(0, 0) \equiv x_0$.
Then, by using Lemma \ref{lem1.7}, we obtain the uniform (in $x_0 \in D$ and $p_0 \in P$) estimate
\begin{align}\label{eq3.12}
\mathbb{E}
\|\chi(t, \theta)\|^{2}
\leq \frac{4B_{R_*}^g N^{2\alpha}}{[\Gamma(\alpha+1)]^{2}}
+2\sum_{i=1}^{d}M_{3,i}N^{2\alpha+2H_{i}-2}
+2R_*^{2}
\end{align}
uniformly in $\theta \in [0, N]$ for all $N > 0$, $\varrho\geq \tilde{\varrho}$ and \( t \geq  \tilde{t}_{\tilde{\varrho}, R}\). Furthermore, we have
\begin{align}
&\|\chi(t, \cdot)\|^{2}_\alpha
= \mathbb{E}\|\chi(t, 0)\|^{2}
+\sum_{N=1}^\infty \frac{1}{2^N N^{2\alpha}}
\mathbb{E}\|\chi(t, \theta)\|_{N}^{2}\nonumber\\
&\leq
R_*^{2} + \sum_{N=1}^\infty \frac{1}{2^N N^{2\alpha} }
\left[\frac{4B_{R_*}^g N^{2\alpha}}{[\Gamma(\alpha+1)]^{2}}
+2\sum_{i=1}^{d}M_{3,i}N^{2\alpha+2H_{i}-2}
+2R_*^{2}\right]\nonumber
\end{align}
for all $x_0 \in D$, $p_0 \in P$, $\varrho\geq \tilde{\varrho}$ and \( t \geq  \tilde{t}_{\tilde{\varrho}, R}\). Therefore,
\begin{align}\label{eq3.12-11}
\|\chi(t, \cdot)\|^{2}_\alpha
\leq 3R_*^{2} +
\frac{4B_{R_*}^g }{[\Gamma(\alpha+1)]^{2}}
+2\sum_{i=1}^{d}M_{3,i} := \hat{R}_*^{2}
\end{align}
for all $x_0 \in D$, $p_0 \in P$, $\varrho\geq \tilde{\varrho}$ and \( t \geq  \tilde{t}_{\tilde{\varrho}, R}\).
In view of $\chi(t, \theta) = (T_{t,0}(x_0 e^{-\varrho\cdot}, p_0))(\theta)$, we obtain that for this $\mathcal{D} $, there exists a time $\tilde{t}_{\varrho, \mathcal{D}}=\tilde{t}_{\varrho, R}>0$ such that
\[
\sup_{x \in D, p \in P}
\|T_{t,0}(x e^{-\varrho\cdot}, p)\|^{2}_\alpha
\leq \hat{R}_*^{2},
\]
whenever $\varrho>\tilde{\varrho}$ and $t \geq \tilde{t}_{\tilde{\varrho}, \mathcal{D}}$.

Define the closed and bounded subset \(\mathfrak{B}^*\) of \(\mathfrak{C}_w\) by
\[
\mathfrak{B}^* := \left\{ \chi(t, \cdot) \in \mathfrak{C}_w :
\|\chi(t, \cdot)\|^{2}_\alpha
\leq \hat{R}_*^{2} \right\},
\]
where the constant $\hat{R}_*^{2}$ is defined in \eqref{eq3.12-11}. Then the set $\mathfrak{B}^*$ uniformly (in \(p_0 \in P\)) absorbs \(\mathcal{D}\) under \(T\) for all $\varrho>\tilde{\varrho}, t \geq \tilde{t}_{\tilde{\varrho},\mathcal{D}}$.
Moreover, since $\mathfrak{B}^*$ is independent of the particular choice of \(\mathcal{D}\), it constitutes a uniformly \(\mathscr{D}\)-absorbing set of \(T\).

Consequently,  \(\mathfrak{B}^* \times P\) is a \(\mathbb{D}\)-absorbing set for the mean-square semi-dynamical system \(\Pi \) in the product space \(\mathfrak{C}_w \times P\).

\subsection{Uniform \(\mathscr{D}\)-asymptotic compactness of $T$ in $\mathfrak{C}_w$}

We shall establish the uniform \(\mathscr{D}\)-asymptotic compactness of \(T\) in \(\mathfrak{C}_w\) via
\cite[Theorem 4]{Caraballo-Morillas-2014} or \cite[Lemma 16]{Wang-Sui}.
This will be done through the following steps.

\subsubsection{Equi-Lipschitz continuity of solution}

Recall that for any $\mathcal{D} \in \mathscr{D}$, there exists a bounded set $D\subset L^{2}(\Omega, \mathcal{F}_{0};\mathbb{R}^d)$ such that $\mathcal{D} = \{xe^{-\varrho\cdot} : x \in D\}$. Without loss of generality, we assume
$\sup_{x \in D} \mathbb{E}\|x\|^{2} \leq R^{2} \text{ with } R \geq R_*.$
Now consider sequences $\{x_{0,n}\}_{n \in \mathbb{N}} \subset D$ and $\{p_{0,n}\}_{n \in \mathbb{N}} \subset P$. Define $\chi_n(t, \theta) := (T_{t,0}(x_{0,n}e^{-\varrho\cdot}, p_{0,n}))(\theta)$, so that in particular
$\chi_n(0, 0) \equiv x_{0,n}.$
We shall show that the family $\{ \chi_n(t, \theta) \}_{n \in \mathbb{N}}$ is equi-Lipschitz continuous in $\theta\geq0$ for all $t \geq 0$.

\begin{lemma}
Let $D\subset L^{2}(\Omega, \mathcal{F}_{0};\mathbb{R}^d)$ be a bounded set.
For sequences $\{x_{0,n}\}_{n \in \mathbb{N}} \subset D$ and $\{p_{0,n}\}_{n \in \mathbb{N}} \subset P$, the sequence  $\{ \chi_n(t, \theta) \}_{n \in \mathbb{N}}$ defined above is equi-Lipschitz continuous in  $\theta\geq0$ for all $t \geq 0$.
\end{lemma}

\begin{proof}
In view of \eqref{eq3.1} and $\chi_n(t, \theta) = (T_{t,0}(x_{0,n}e^{-\varrho\cdot}, p_{0,n}))(\theta)$, one can see that
\begin{align}
&\mathbb{E}\|\chi_n(t, \theta)-\chi_n(t, \theta+\Delta)\|^{2}
\nonumber\\
&=\mathbb{E}\|(T_{t,0}(x_{0,n} e^{-\varrho\cdot}, p_{0,n}))(\theta)
-(T_{t,0}(x_{0,n} e^{-\varrho\cdot}, p_{0,n}))(\theta+\Delta)\|^{2}
\nonumber\\
&\leq 5\mathbb{E}\|x_{0,n}\|^{2}
[e^{-\varrho(t+\theta)}-e^{-\varrho(t+\theta+\Delta)}]^{2}\nonumber\\
&+\frac{5}{[\Gamma(\alpha)]^{2}}
\mathbb{E}\bigg\|
\int_0^t[(t+\theta-s)^{\alpha-1}
-(t+\theta+\Delta-s)^{\alpha-1}]
e^{-\varrho(t+\theta-s)}
g(x(s,x_{0,n},p_{0,n}), \vartheta_s(p_{0,n}))ds\bigg\|^{2}
\nonumber\\
&+\frac{5}{[\Gamma(\alpha)]^{2}}
\mathbb{E}\bigg\|\int_0^t
(t+\theta+\Delta-s)^{\alpha-1}
[e^{-\varrho(t+\theta-s)}-e^{-\varrho(t+\theta+\Delta-s)}]
g(x(s,x_{0,n},p_{0,n}), \vartheta_s(p_{0,n}))ds\bigg\|^{2}
\nonumber\\
&+\frac{5}{[\Gamma(\alpha)]^{2}}
\mathbb{E}\bigg\|
\int_0^t[(t+\theta-s)^{\alpha-1}
-(t+\theta+\Delta-s)^{\alpha-1}]
e^{-\varrho(t+\theta-s)}dB^{H,\lambda}(s)\bigg\|^{2}
\nonumber\\
&+\frac{5}{[\Gamma(\alpha)]^{2}}
\mathbb{E}\bigg\|
\int_0^t(t+\theta+\Delta-s)^{\alpha-1}
[e^{-\varrho(t+\theta-s)}-e^{-\varrho(t+\theta+\Delta-s)}]
dB^{H,\lambda}(s)\bigg\|^{2}\nonumber\\
&:=5\mathbb{E}\|x_{0,n}\|^{2}J_{4}+
\frac{5}{[\Gamma(\alpha)]^{2}}[J_{5}+J_{6}+J_{7}+J_{8}].\nonumber
\end{align}
It is easy to obtain that
\begin{align}
J_{4}\leq \varrho^{2} e^{-2\varrho \theta}\Delta^{2}.\nonumber
\end{align}
Taking into account \eqref{eq3.10}, we have
\begin{align}
J_{5}&\leq
\int_0^t\int_0^t
[(t+\theta-s)^{\alpha-1}
-(t+\theta+\Delta-s)^{\alpha-1}]
[(t+\theta-r)^{\alpha-1}
-(t+\theta+\Delta-r)^{\alpha-1}]
\nonumber\\
&
~~\times\mathbb{E}\|g(x(s,x_{0,n},p_{0,n}), \vartheta_s(p_{0,n}))\|^{2}dsdr
\nonumber\\
&\leq B_{R}^{g}
\bigg[\int_0^t[(t+\theta-s)^{\alpha-1}
-(t+\theta+\Delta-s)^{\alpha-1}]ds\bigg]
\nonumber\\
&\leq \frac{4B_{R}^{g}\Delta^{2\alpha}}{\alpha^{2}}.\nonumber
\end{align}
Analogously, it holds that
\begin{align}
J_{6}&\leq B_{R}^{g}(1-e^{-\varrho\Delta}  )^{2}
\bigg[\int_0^t(t+\theta+\Delta-s)^{\alpha-1}
e^{-\varrho(t+\theta-s)}ds\bigg]^{2}\nonumber\\
&\leq B_{R}^{g}[\Gamma(\alpha)]^{2}\varrho^{2-2\alpha}\Delta^{2}.\nonumber
\end{align}
From the It\^{o} isometry of TFBM \cite[(2.12)]{Zhang-Wang-Hu}, we deduce that
\begin{align}
J_{7}&\leq\sum_{i=1}^{d}c_{H_{i}}\int_0^{t}\int_0^{t}
e^{-\varrho (t + \theta-s)}
e^{-\varrho (t + \theta-r)}
[(t+\theta-s)^{\alpha-1}
-
(t+\theta+\Delta-s)^{\alpha-1}]
\nonumber\\
&~~\times
[(t+\theta-r)^{\alpha-1}
-
(t+\theta+\Delta-r)^{\alpha-1}]
|s-r|^{2H_{i}-2}dsdr\nonumber\\
&\leq \sum_{i=1}^{d}c_{H_{i}}\Delta^{2\alpha+2H_{i}-2}
\int_0^{\infty}\int_0^{\infty}
e^{-\varrho \Delta(r+s)}
[s^{\alpha-1}-(s+1)^{\alpha-1}]
\nonumber\\
&~~\times
[r^{\alpha-1}-(r+1)^{\alpha-1}]
|s-r|^{2H_{i}-2}dsdr
\nonumber\\
&\leq \sum_{i=1}^{d}c_{H_{i},\varrho,\alpha}\Delta^{2\alpha+2H_{i}-2}.
\nonumber
\end{align}
Similarly, we derive
\begin{align}
J_{8}&\leq(1-e^{-\varrho\Delta})^{2}\sum_{i=1}^{d}c_{H_{i}}
\int_0^{t}\int_0^{t}
e^{-\varrho (t + \theta-s)}
e^{-\varrho (t + \theta-r)}\nonumber\\
&~~\times(t+\theta+\Delta-s)^{\alpha-1}
(t+\theta+\Delta-r)^{\alpha-1}
|s-r|^{2H_{i}-2}dsdr\nonumber\\
&\leq\Delta^{2}\sum_{i=1}^{d}c_{H_{i},\alpha,\varrho}.
\nonumber
\end{align}
Therefore, we conclude that
\begin{align}\label{eq3.13}
&\mathbb{E}\|\chi_n(t, \theta)-\chi_n(t, \theta+\Delta)\|^{2}\nonumber\\
&\leq 5\mathbb{E}\|x_{0,n}\|^{2}
\varrho^{2} e^{-2\varrho \theta}\Delta^{2}
+c_{\alpha}B_{R}^{g} \Delta^{2\alpha}\nonumber\\
&
+\sum_{i=1}^{d}c_{H_{i},\alpha,\varrho}B_{R}^{g}
\Delta^{2}+\sum_{i=1}^{d}c_{H_{i},\alpha,\varrho}
\Delta^{2\alpha+2H_{i}-2},
\end{align}
which implies that $\chi_n(t, \theta)$ is equi-Lipschitz continuous in $\theta\geq0$ for all $t \geq 0$.

\end{proof}

\subsubsection{Existence of \(\chi^*(\theta)\) for \(\theta \geq 0\)}

In view of the uniform estimate \eqref{eq3.12}, we have
\begin{align}\label{eq3.14}
\mathbb{E}
\|\chi(t, \theta)\|^{2}
\leq \frac{4B_{R_*}^g N^{2\alpha}}{[\Gamma(\alpha+1)]^{2}}
+2\sum_{i=1}^{d}M_{3,i}N^{2\alpha+2H_{i}-2}
+2R_*^{2}
\end{align}
uniformly in $\theta \in [0, N]$ for all $N > 0$, $\varrho>\tilde{\varrho}$ and $t \geq \tilde{t}_{\tilde{\varrho}, R}$. In addition, by estimate \eqref{eq3.13}, it holds that
\begin{align}
\mathbb{E}\|\chi_n(t, \theta)-\chi_n(t, \theta+\Delta)\|^{2}
&\leq 5\mathbb{E}\|x_{0,n}\|^{2}
\varrho^{2} \Delta^{2}
+c_{\alpha}B_{R}^{g} \Delta^{2\alpha}\nonumber\\
&
+\sum_{i=1}^{d}c_{H_{i},\alpha,\varrho}B_{R}^{g}
\Delta^{2}+\sum_{i=1}^{d}c_{H_{i},\alpha,\varrho}
\Delta^{2\alpha+2H_{i}-2},\nonumber
\end{align}
uniformly in $\theta \in [0, N]$ for all $N > 0$, $\varrho>\tilde{\varrho}$ and $t \geq \tilde{t}_{\tilde{\varrho}, R}$.
As a result, the sequence $\chi_n(t, \cdot)$ is equi-Lipschitz continuous and uniformly bounded on each interval $\theta \in [0, N]$.
This allows us to apply either \cite[Theorem 4]{Caraballo-Morillas-2014} or \cite[Lemma 16]{Wang-Sui} in the space $C([0, N], L^{2}_w(\Omega, \mathcal{F};\mathbb{R}^d))$ of continuous functions $f: [0, N] \to L^{2}_w(\Omega, \mathcal{F};\mathbb{R}^d)$.
Accordingly, there exist a function
$\chi^* \in C([0, N], L^{2}_w(\Omega, \mathcal{F};\mathbb{R}^d))$,
a subsequence $\{\chi_{n_{k}}\}$ of $\{\chi_{n}\}$ and a sequence $t_k \to \infty$ (as $k\to \infty$) such that
$\chi_{n_{k}}(t_k, \cdot)\rightarrow \chi^*(\cdot)$ in $C([0, N], L^{2}_w(\Omega, \mathcal{F};\mathbb{R}^d))$. More precisely, for every $\theta_k
\to \theta\in [0,N]$,
\[\chi_{n_{k}}(\theta_k):=
\chi_{n_{k}}(t_k,\theta_k)
\to \chi^*(\theta)
\]
in $L^{2}_w(\Omega, \mathcal{F}; \mathbb{R}^d)$ as $k\to \infty$.
Finally, by choosing \(N\) sufficiently large and applying a diagonal argument, we obtain that   \(\chi^*(\theta)\) is well-defined for every \(\theta \geq 0\).

From the identity $\chi_n(t, \theta) = (T_{t,0}(e^{-\varrho \cdot}x_{0,n}, p_{0,n}))(\theta)$,  we immediately derive that
for any bounded sequences \(\{x_{0,n}\}_{n \in \mathbb{N}}\subset L^{2}(\Omega, \mathcal{F}_{0}; \mathbb{R}^d)\), \(\{p_{0,n}\}_{n \in \mathbb{N}}\subset P\) and \(t_n \to \infty\), the sequence \(\{T_{t_n,0}(e^{-\varrho \cdot}x_{0,n}, p_{0,n})\}_{n \in \mathbb{N}}\) possesses a convergent subsequence such that
\[(T_{t_k,0}(e^{-\varrho \cdot}x_{0,n_{k}}, p_{0,n_{k}}))(\cdot) = \chi_{n_{k}}(t_k, \cdot) \to \chi^*(\cdot)\]
in $C([0, \infty), L^{2}_w(\Omega, \mathcal{F};\mathbb{R}^d))$. This implies that
the cocycle \(T\) is uniformly \(\mathscr{D}\)-asymptotically compact in $C([0, \infty), L^{2}_w(\Omega, \mathcal{F};\mathbb{R}^d))$.

Finally, an application of Lemma 5 in \cite{Cui-Kloeden-2024} yields the existence of the mean-square attractors.

\begin{theorem}\label{thm1.5}
Let Assumptions 1-2 hold. Then the cocycle \( T \) has a mean-square uniform \( \mathscr{D} \)-attractor \( \mathcal{A}_P \) in \( \mathfrak{C}_{w} \) given by
\[
    \mathcal{A}_P = \operatorname{cl}_{\mathfrak{C}_{w}} \left\{ \chi \in \mathfrak{C}_{w} \;\middle|\;
    \begin{aligned}
         &\exists\,\{e^{-\varrho \cdot}x_{n}\}_{n \in \mathbb{N}}  \subset \mathcal{D} \text{ for a } \mathcal{D} \in \mathscr{D},\;
        \{p_n\}_{n \in \mathbb{N}}\subset P \text{ and } \\
        &t_n \to \infty \text{ such that } T_{t_n,0}(e^{-\varrho \cdot}x_{n}, p_{n}) \to \chi \text{ in}~\mathfrak{C}_{w}
    \end{aligned}
    \right\},
\]
and the semi-group \( \Pi \) has a mean-square \( \mathbb{D} \)-attractor \( \mathbb{A} \) in \( \mathbb{X} = \mathfrak{C}_{w} \times P \) given by
\[
    \mathbb{A} = \operatorname{cl}_{\mathbb{X}} \left\{ (x, p) \in \mathbb{X} \;\middle|\;
    \begin{aligned}
        &\exists\, \{x_n\}_{n \in \mathbb{N}} \text{ bounded},\; \{p_n\}_{n \in \mathbb{N}} \subset P \text{ and } t_n \to \infty \\
        &\text{ such that } \Pi_{{t_n},0}(x_n e^{-\varrho\cdot}, p_n) \to (x, p) \text{ in}~\mathfrak{C}_{w} \times P
    \end{aligned}
    \right\}.
    \]
\end{theorem}

Indeed, since \(P\) is compact, the uniform \(\mathscr{D}\)-asymptotic compactness of \(T\) in \(\mathfrak{C}_w\) established above implies that the skew-product semi-flow \(\Pi\) is \(\mathbb{D}\)-asymptotically compact on \(\mathfrak{C}_{w} \times P\).
In addition, since \(T\) admits a closed, bounded and uniformly \(\mathscr{D}\)-absorbing set \(\mathcal{B}^*\) in \(\mathfrak{C}_{w}\), the skew-product semi-flow \(\Pi\) also possesses a closed, bounded \(\mathbb{D}\)-absorbing set given by \(\mathcal{B}^* \times P\). Consequently, by Lemma 5 in \cite{Cui-Kloeden-2024},
\(\Pi\) admits a mean-square \(\mathbb{D}\)-attractor \(\mathbb{A}\) in \(\mathfrak{C}_{w} \times P\). The existence of the mean-square uniform \(\mathscr{D}\)-attractor \(\mathcal{A}_P\) of \(T\) follows by a similar argument.

\section*{Acknowledgments}

The authors are grateful to Professor Hongyong Cui and Professor Peter Kloeden for their valuable comments and constructive suggestions, which have helped enhance the quality of this paper.

\end{document}